\documentclass[12pt,a4paper]{amsart}
\usepackage{amsfonts}
\usepackage{amscd}
\usepackage[latin1]{inputenc}
\usepackage{t1enc}
\usepackage[mathscr]{eucal}
\usepackage{indentfirst}
\usepackage{graphicx}
\usepackage{graphics}
\usepackage{pict2e}
\usepackage{epic}
\numberwithin{equation}{section}
\usepackage[margin=2.6cm]{geometry}
\usepackage{epstopdf} 
\usepackage{amsmath}
\allowdisplaybreaks
\usepackage{amssymb}
\usepackage{amsthm}
\usepackage{bbm}
\usepackage{tikz}
\usepackage{paralist}
\usepackage[onehalfspacing]{setspace}
\usepackage{amsfonts}
\usepackage{color}
\usepackage{nicefrac}
\usepackage{mathrsfs}
\usepackage{multirow}

\newcounter{dummy}
\usepackage{hyperref}
\usepackage{enumitem}
\makeatletter
\newcommand\myitem[1][]{\item[#1]\refstepcounter{dummy}\def\@currentlabel{#1}}
\makeatother
\newtheorem{thm}{Theorem}
\numberwithin{thm}{section}
\newtheorem{lem}[thm]{Lemma}
\newtheorem{prop}[thm]{Proposition}
\newtheorem{defi}[thm]{Definition}
\newtheorem{coro}[thm]{Corollary}
\newtheorem{quest}[thm]{Question}

\newtheorem*{thm*}{Theorem}
\newtheorem*{prop*}{Proposition}
\numberwithin{equation}{section}
\theoremstyle{remark}

\newcommand{\A}{\mathscr{A}}
\newcommand{\B}{\mathscr{B}}
\newcommand{\E}{\mathcal{E}}

\newcommand{\V}{\vert}

\newcommand{\R}{\mathbb{R}}
\newcommand{\N}{\mathbb{N}}

\newcommand{\C}{\mathbb{C}}

\newcommand{\Q}{\mathscr{Q}}

\newcommand{\dx}{\,\textup{d}x}

\renewcommand{\phi}{\varphi}

\everymath{\displaystyle}

\DeclareMathOperator{\Lin}{Lin}

\DeclareMathOperator{\Image}{Im}
\DeclareMathOperator{\curl}{curl}

\DeclareMathOperator{\spann}{span}
\DeclareMathOperator{\divergence}{div}

\definecolor{Gump}{rgb}{0,0.6,0.4}
\definecolor{Hanks}{rgb}{0.7,0.3,0.1}

\begin{document}
\title[Complex constant rank condition]{On the complex constant rank condition and inequalities for differential operators}
\author[Schiffer]{Stefan Schiffer}
\address{Max-Planck Institute for Mathematics in the Sciences} 
\email{stefan.schiffer@mis.mpg.de}
\begin{abstract}
In this note, we study the complex constant rank condition for differential operators and its implications for coercive differential inequalities. These are inequalities of the form 
\[
\Vert A u \Vert_{L^p} \leq \Vert \mathscr{A} u \Vert_{L^q},
\]
for exponents $1\leq p,q <\infty$ and homogeneous constant-coefficient differential operators $A$ and $\mathscr{A}$. The functions $u \colon \Omega \to \R^d$ are defined on open and bounded sets $\Omega \subset \R^N$ satisfying certain regularity assumptions. Depending on the order of $A$ and $\mathscr{A}$, such an inequality might be viewed as a generalisation of either Korn's or Sobolev's inequality, respectively. In both cases, as we are on bounded domains, we assume that the Fourier symbol of $\mathscr{A}$ satisfies an algebraic condition, the complex constant rank property.

\end{abstract}
\maketitle
\section{Introduction}
\subsection{Motivation and a role-model result}
Inequalities featuring differential inequalities undeniably play a huge role in various aspects of PDEs. The type of inequalities studied in this note are stated for functions $u \colon \Omega \to \R^d$ for open and bounded domains $\Omega$ which satisfy certain regularity assumptions. The inequalities roughly fit into one of the following two types of inequalities (eg. \cite{Strauss,CFM,VS}):
\begin{enumerate}[label=(\alph*)]
    \item \emph{Korn-type inequalities:} These are inequalities that feature a constant coefficient operator $\A$ of order $k$ and estimate 
        \begin{equation} \label{Korn:ex}
            \Vert D^k u \Vert_{L^p} \leq C\Vert \A u \Vert_{L^p}
        \end{equation}
    for all $u$ in a suitable functional space $X$, for example $X= W^{k,p}_0(\Omega;\R^d)$, where $\Omega \subset \R^N$ is an open and bounded domain;
    \item \emph{Sobolev-type inequalities:} Again, taking a differential operator $\A$, such an inequality takes the form
        \begin{equation} \label{Sobolev:ex}
                \Vert D^{k-1} u \Vert_{L^q} \leq C \Vert \A u \Vert_{L^p}
        \end{equation}
    for certain exponents $p$ and $q$ and $u \in W^{k,p}(\R^N;\R^d)$ (or $u \in W^{k,p}(\Omega;\R^d)$).
\end{enumerate}

Obviously, equations \eqref{Korn:ex} and \eqref{Sobolev:ex} \emph{do not} cover the most general forms of those inequalites and, moreover, one also may combine Korn- and Sobolev-type inequalities to get so-called Korn-Maxwell-Sobolev inequalities (e.g. \cite{Neff1,GS,GLN22}).
Both inequalities \eqref{Korn:ex}, \eqref{Sobolev:ex} and their respective generalisations and extensions have been extensively studied in a vast number of works; a by no means exhaustive overview is given below. To motivate the main question answered in this note, the following Proposition \ref{propintro} displays some of the previous results regarding inequalities \eqref{Korn:ex} and \eqref{Sobolev:ex}.

\smallskip
Consider a differential operator $\A$ of order $k$ of the form 
\begin{equation} \label{intro:A}
    \A u = \sum_{\vert \alpha \vert =k} \A_{\alpha} \partial_{\alpha} u,
\end{equation}
where $\A_{\alpha} \in \Lin(\R^d;\R^l)$ are linear maps (independent of the spatial variable). The \emph{Fourier symbol} of $\A$ is given by 
\begin{equation}
    \A [\xi] = \sum_{\vert \alpha \vert=k} \A_{\alpha} \xi^{\alpha} \in \Lin(\R^d;\R^l), \quad \xi \in \R^N \setminus \{0\}.
\end{equation}
Note that the Fourier symbol is a polynomial in $\xi$; it can be therefore defined as an element of $\Lin(\C^d;\C^l)$ by the obvious extension for all $\xi \in \C^N$.
The following statement is a summary results known for the very basic  forms of  inequalities \eqref{Korn:ex} and \eqref{Sobolev:ex}.
\begin{prop} \label{propintro}
Let $\A$ be as in \eqref{intro:A} and let $\Omega$ be a bounded Lipschitz domain. \begin{enumerate} [label=(\alph*)]
    \item \label{propintro:a} Inequality \eqref{Korn:ex} is true for $X=W_0^{k,p}(\Omega;\R^d)$, $1<p<\infty$, if and only if $\A$ is \textbf{elliptic}, i.e. the Fourier symbol is injective for all $\xi \in \R^N \setminus \{0\}$;
    \item \label{propintro:b} Inequality \eqref{Korn:ex} is true for $X=W^{k,p}(\Omega;\R^d) \slash \Q$, $1<p<\infty$, for a finite-dimensional \textbf{quotient} $\Q$, if and only if $\A$ is \textbf{$\C$-elliptic},
     i.e. the Fourier symbol is injective for all $\xi \in \C^N \setminus \{0\}$;
     \item \label{propintro:c} Both statements of \ref{propintro:a} and \ref{propintro:b} are false for $p=1$ and $p=\infty$ (unless $\A = L \circ D^k$ for an injective linear operator $L$); 
     \item \label{propintro:d} Inequality \eqref{Sobolev:ex} is true for $X=W^{k,p}_0(\Omega;\R^d)$ for $1<p<N$ and $q= p*= \tfrac{Np}{N-p}$ if and only if $\A$ is elliptic;
     \item \label{propintro:e} Inequality  \eqref{Sobolev:ex} is true for $X=W^{k,1}(\R^N;\R^d)$ and $q=1*= \tfrac{N}{N-1}$ if and only if $\A$ is elliptic and \textbf{cancelling}, meaning that 
        \[
            \bigcap_{\xi \in \R^N \setminus \{0\}} \Image \A [\xi] = \{0\}.
        \]
\end{enumerate}
\end{prop}
Indeed, \ref{propintro:a} is a consequence of Calderon-Zygmund theory/multiplier theorems (e.g. \cite{CZ,CZ2,FM,GuRa}). The inequality corresponding to \ref{propintro:b} relies on results on $\C$-elliptic operators \cite{Kala,Smith} and is proven there and in \cite{DG}. Note that the inequality being valid in $W^{k,p}(\Omega;\R^d) /\Q$ is equivalent to
\[
\inf_{v \in \Q} \Vert D^k(u-v) \Vert_{L^p} \leq \Vert \A u \Vert_{L^p};
\]
if for example $k=1$ and $\A$ is the \emph{symmetric gradient}, then $\Q$ is the space of linear maps $v= Bx$ for skew-symmetric matrices $B$.

Statement \ref{propintro:c} is commonly known as \emph{Ornstein's non-inequality}, cf. \cite{Ornstein} for Ornstein's original treatment and \cite{CFM,KK0, KirchKris, FG} for more recent developments and improvements.

Having shown Korn's inequality \ref{propintro:a}, 
 the Sobolev-type inequality \ref{propintro:d} is no surprise, it just follows from \ref{propintro:a} and the standard Sobolev embedding theorem. That inequality \eqref{Sobolev:ex} is true for $p=1$ and the critical exponent $N/(N-1)$, i.e. statement \ref{propintro:e}, is, on the contrary,  far from trivial. It has been shown by van Schaftingen in \cite{VS} using a remarkable estimate due to Bourgain \& Brezis (cf. \cite{BB1,BB2}, see also \cite{VS2,SVS,Spector,GS,VSG} for some recent results). In its original form, this estimate states that the inequality 
\[
 \int f \phi \dx  \leq C \Vert f\Vert_{L^1} \Vert \nabla \phi \Vert_{L^N} 
\]
is true for any $\phi \in C_c^{\infty}(\R^N;\R^N)$ and all $f \in L^1(\R^N;\R^N)$ that satisfy the additional constraint $\divergence f=0$. A modification of this estimate will also play an important role in the proof of the second main theorem which is stated below. We shall also mention that this estimate is far from trivial if we consider divergence-free $L^1$-functions on domains $\Omega$ instead; therefore, as in the works \cite{GmR,GS}, we need to be even more careful with the regularity of $\Omega$.
\subsection{Main results} 
The statements of Proposition \ref{propintro} might be viewed as a rather harsh restriction on the operator $\A$ (not many operators are elliptic or even $\C$-elliptic). Suggesting another viewpoint, these restrictions are so harsh because the condition for the left-hand-side of inequalities \eqref{Korn:ex} and \eqref{Sobolev:ex} is rather weak- for example, \eqref{Korn:ex} might be viewed as 
    \begin{equation} \label{Korn:2}
        \Vert A u \Vert_{L^p} \leq C \Vert \A u \Vert_{L^p}
    \end{equation}
for \emph{any} differential operator $A$ of order $k$. The left-hand-side $A=D^k$ correpsonds to the worst case scenario for inequality \eqref{Korn:2}.

\begin{quest}
Given a differential operator $\A$, can we characterise differential operators $A$ such that inequality \eqref{Korn:2} (or a suitable version of \eqref{Sobolev:ex}) is true?
\end{quest}

As it was noted in multiple contexts, the correct replacement for of being \emph{elliptic} often is the so called \emph{constant rank condition}, cf. \cite{Murat,Wilcox,FM}. That is, the rank of the Fourier symbol $\A[\xi]$ is constant along all $\xi \in \R^N \setminus \{0\}$. In this work, however, we are more interested in inequalities on domains without imposing further boundary conditions (i.e. Proposition \ref{propintro} \ref{propintro:b}). Therefore, we work with the corresponding replacement/weakening of $\C$-ellipticity (cf. \cite{GSPoinc}).
\begin{defi}
We say that a differential operator $\A$ as in \eqref{intro:A} satisfies the \textbf{complex constant rank condition} (or the constant rank property with respect to $\C$), if there is $r \in \N$ such that for all $\xi \in \C^N \setminus \{0\}$
\[
\dim_{\C} \ker_{\C} \A[\xi] =r.
\]
\end{defi}
The purpose of this work is to provide a proof of the following two results. In both of these, $\Omega \subset \R^N$ is a bounded domain enjoying an additional regularity assumptions. It is possible to drop the assumption, that $\Omega$ is \emph{connected}; as we then get inequalities on finitely many connected components. For simplicity, however, we stick to the case of domains for the remainder of this work.

\begin{thm}[Korn-type inequality] \label{thm1}
Suppose that $\Omega \subset \R^N$ is a Lipschitz domain. Suppose that $\A$ is a constant-coefficient differential operator of order $k$ satisfying the complex constant rank property and that $A$ also has constant coefficients and order $k$. Let $1<p<\infty$. Then the following two statements are equivalent:
\begin{enumerate}[label=(\alph*)]
    \item \label{thm1:1} There exists a finite dimensional subspace $\Q \subset W^{k,p}(\Omega;\R^d)$ such that 
        \[
            \inf_{v \in \Q} \Vert A(u-v) \Vert_{L^p} \leq C_p \Vert \A u \Vert_{L^p}
        \]
    for all $u \in W^{k,p}(\Omega;\R^d)$;    

    \item \label{thm1:2} for all $\xi \in \C^N \setminus \{0\}$ we have $\ker \A[\xi] \subset \ker A[\xi]$.
\end{enumerate}
 \end{thm}
Fixing a suitable regularity for $\Omega$ in the setting of the Sobolev type inequality (in particular $p=1$) is more challenging. For now, let us give a concise definition of the required assumption, we comment further on that in Section \ref{remarks}. 

Denote by $M_k = \tbinom{N+k-1}{N}$ and by  $\R^{M_k}$ the Euclidean space spanned up by $e_{\beta}$, $\beta \in \N^n$, $\V \beta \V =k$. Then we define by 
$\divergence^k \colon C^{\infty}(\R^N;(\R^{M_k}) \to C^{\infty}(\R^N)$ the $k$-fold divergence operator;
\[
\divergence^k u = \sum_{\V \beta \V=k} \partial_{\beta} u_{\beta}.
\]
\begin{defi} \label{def:DFED}
We say that $\Omega \subset \R^N$ is a divergence-free-extension domain if there exists a finite dimensional subspace $V_k \subset L^{\infty}(\Omega;\R^{M_k}) \subset L^1(\Omega;\R^{M_k})$, a bounded projection map $P_k \colon L^1(\Omega;\R^{M_k}) \to V_k$ and a bounded linear \textbf{extension operator} 
\[
\mathcal{E} \colon L^1(\Omega;\R^{M_k}) \to L^1(\R^N;\R^{M_k})
\]
such that \begin{enumerate} [label=(\alph*)]
    \item The operator is an extension up to the vector space $V_k$, i.e. 
    \[
    \E u - u = P_k u \quad \text{in } \Omega;
    \]
    \item The operator preserves solenoidality, i.e.
    \[
\divergence^k u =0 \text{ in } \mathcal{D}'(\Omega)
\quad \Longrightarrow \quad \divergence^k \E u =0 \text{ in } \mathcal{D}'(\R^N).
\]
\end{enumerate}
\end{defi}
As mentioned, we comment further on this requirement in Section \ref{remarks}. It is, however, important to stress that a suitable sharp classification of such domains is, to the best of the author's knowledge, an open question and beyond the scope of this article; it is however plausible that at least Lipschitz domains obey the divergence-free extension property (cf. \cite[Open Question 3.4]{GS}).

\begin{thm}[Sobolev-type inequality] \label{thm2}
Let $\Omega$ be a divergence-free-extension-domain. Suppose that $\A$ is a constant-coefficient differential operator of order $k$ satisfying the complex constant rank property and that $A$ also has constant coefficients and order $k-1$. Let $1\leq p<N$ and $p*= \tfrac{Np}{N-p}$. Then the following two statements are equivalent:
\begin{enumerate}[label=(\alph*)]
    \item \label{thm2:1} There exists a finite dimensional subspace $\Q \subset W^{k,p}(\Omega;\R^d)$ such that 
        \[
            \inf_{v \in \Q} \Vert A(u-v) \Vert_{L^{p*}} \leq C_p \Vert \A u \Vert_{L^p}
        \]
    for all $u \in W^{k,p}(\Omega;\R^d)$    
    \item \label{thm2:2} for all $\xi \in \C^N \setminus \{0\}$ we have $\ker \A[\xi] \subset \ker A[\xi]$.
\end{enumerate}
\end{thm}

\subsection{Organisation of the paper}
We shall now give a brief overview over the remainder of this paper. The following Section \ref{remarks} is concerned with a few remarks regarding the main theorems \ref{thm1} and \ref{thm2} and possible variations. In Section \ref{prelims} we revisit some important results that allow a short proof of both Theorem \ref{thm1} and Theorem \ref{thm2}. The proofs of both main theorems are then carried out in the last Section \ref{proofs}.

\subsection*{Acknowledgements}
The author would like to thank Franz Gmeineder for enlightening discussions. The research of this paper was mainly conducted while the author was supported by the Deutsche Forschungsgemeinschaft (DFG, German Research Foundation) through the graduate school BIGS of the Hausdorff Center for Mathematics (GZ EXC 59 and 2047/1, Projekt-ID 390685813). 
\section{Some comments on the main theorems} \label{remarks}
Before reminding the reader of some auxiliary results, let us briefly comment on the results achieved in Theorems \ref{thm1} and \ref{thm2}. First of all, note that operators like the gradient, higher dimensional gradients, the rotation $\curl$ and the divergence $\divergence$ all satisfy the constant rank property over $\C$. On the contrary, the Laplacian and the Cauchy-Riemann operator, albeit being $\R$-elliptic, do \emph{not} enjoy the complex constant rank condition.
\subsection{Remarks on the Korn-type inequality}
\begin{enumerate}[label=(\alph*)]
    \item As mentioned before, Theorem \ref{thm1} \emph{cannot} hold in both limiting cases $p=1$ and $p=\infty$, unless the entries of the differential operator $A$ can be written as a linear combination of the entries of $\A$, cf. \cite{Ornstein,DLM}. Some modern approaches and proofs of this statement can be found in \cite{CFM,KirchKris,FG}.
    \item As Korn's inequality is not true in \emph{any} open and bounded domain, also Theorem \ref{thm1} cannot be true for domains with irregular boundary. As far as the author is aware, Theorem \ref{thm1} should be true whenever $\Omega$ is a John domain (cf. \cite{Hurri,ADM,DRS} for more details), which in particular includes any domain with Lipschitz boundary (also see Lemma \ref{lemma1}).
    \item For general $\A$ that do not satisfy the complex constant rank property, the statement of Theorem \ref{thm1} is no longer valid; the second condition needs refinement. In particular, taking $A=D^2 \circ \Delta$ and $\A=\Delta^2$ gives a counterexample to '\ref{thm1:2} $\Rightarrow$ \ref{thm1:1}'.
    \item Once having established Theorem \ref{thm1}, we may fully characterise a suitable quotient space $\Q$. Indeed, using elementary functional analytic methods (e.g. uniform convexity of $L^p$-spaces), we can show that the inequality works whenever $\ker \A \subset \ker A + \Q$.
    \item  Theorem \ref{thm1} may be viewed as the suitable replacement of the corresponding theorem for zero boundary values, i.e. for $\A$ with constant rank in $\R$ we have
        \[
            \Vert A u \Vert_{L^p} \leq C_p \Vert \A u \Vert_{L^p}
        \]
    for all $u \in W^{k,p}_0(\Omega;\R^d)$ if and only if $\ker \A[\xi] \subset \ker A[\xi]$. This can be attributed to the fact that the constant rank condition in $\R$ allows the usage of multiplier theorems, cf. \cite{FM,Raita,GuRa,VS}. On the contrary, in our situation without zero boundary values, only considering real Fourier modes is insufficient. Heuristically speaking, it is much more convenient to view the Fourier symbol $\A[\cdot] \in C^{\infty}(\C^N; \Lin(\C^d;\C^l))$ as a polynomial (instead of just being a smooth map), hence $\C$ is a more appropriate field than $\R$. The proof presented here hence also relies on an algebraic lemma proved in \cite{GSPoinc} that is only valid for the algebraically closed field $\C$, but not for $\R$.
\end{enumerate}
\subsection{Remarks on the Sobolev-type inequality}
\begin{enumerate} [label=(\alph*)]
    \item For $p \neq 1$, Theorem \ref{thm2} is an easy consequence of the classical Sobolev embedding $W^{1,q} \hookrightarrow L^{q^{\ast}}$ for $1 < q< N$ combined with Theorem \ref{thm1}. In that case, we actually can easily go to more general domains and do not need the additional regularity assumpiton. Whereas Korn-type inequalities break down for $p=1$, here this limiting case is included (but is not a trivial consequence of Sobolev embeddings).
    \item In the case $A=\nabla^{k-1}$ (with zero boundary values), van Schaftingen showed the validity of the Sobolev-type inequality if and only if $\A$ is elliptic and cancelling. Using the previously mentioned Fourier methods, cf.  \cite{FM,Raita,GuRa}, one may also show the (real) constant rank version of Theorem \ref{thm2} in $W^{k,p}_0(\Omega;\R^d)$ (similar to the proof of Theorem \ref{thm2} below, also cf. \cite{VSG}).
    \item Theorem \ref{thm2} was proven for $A=D^{k-1}$ by Gmeineder \&  Rai\c{t}\u{a} in \cite{GmR} on a ball. Both their proof and the proof of Theorem \ref{thm2} rely on the $L^1$-estimates given by Bourgain \& Brezis \cite{BB1,BB2} and on the related work of van Schaftingen \cite{VS}.
    \item Note that due to the constant rank condition in $\C$, we do not need to worry about van Schaftingen's cancellation condition in our setting. As it was noted in \cite{GmR}, $\C$-ellipticity implies cancellation and ellipticity. For operators of constant rank in $\C$ we do not exactly  have this implication. Nevertheless, we can show that a Bourgain-Brezis-type estimate still holds (cf. Lemma \ref{lemma5}).
    \item Let us finally comment on the required regularity of the domain, cf. Definition \ref{def:DFED}. It is crucial to stress, that the divergence-free extension property is non-trivial for the limiting case $L^1$ and even more challenging than the $L^p$-case, which was studied in \cite{KMPT}. To date, using e.g. methods of reflection, this property basically is clear for balls and cubes (cf. \cite{GmR,GS} and \cite{BVS} for further discussions). It seems to be likely, that this property at least holds for Lipschitz domains (as conjectured in \cite{GS}).
    \item The finite-dimensional vector space $V_k$ in Definition \ref{def:DFED} attributes to \emph{topological} features of the domain $\Omega$. If the domain $\Omega$ is homeomorphic to a ball, then we expect $V_k=\{0\}$. In contrast, for example for an annulus $B_2(0)\setminus B_1(0)$, we may not be able to extend functions of the form $u(x) = x/\vert x \vert^N$ to be divergence-free on the whole space; therefore we have to take a suitable quotient there.
\end{enumerate}
\section{Auxiliary Results} \label{prelims}
Without further ado, let us give a short list of results that are crucial in the proofs Theorem \ref{thm1} and \ref{thm2}. The first one is an estimate due to Ne\v{c}as connecting a function to its derivative in negative Sobolev spaces \cite{Czech,Nevcas}. Denote by $\mathrm{Pol}_{s+1}$ the space of $\R^d$-valued polynomials on $\R^N$ of order $s+1$.
\begin{lem} \label{lemma1}
Let $\Omega \subset \R^N$ be a domain with Lipschitz boundary. Then for all $s \in \N$ and all $1<p<\infty$, there is a constant $C(s,p)>0$ and bounded linear map $P_s \colon L^p(\Omega;\R^d) \to \mathrm{Pol}_{s+1}$, such that 
    \begin{equation}
        \Vert u -P_s u \Vert_{L^p} \leq C \Vert D^s u \Vert_{W^{-s,p}} \quad \forall u \in L^p(\Omega;\R^d).
    \end{equation}
\end{lem}
The regularity of the boundary might be weakened (cf. \cite{ADM}), but above lemma does \emph{not} hold for any domain. It is also worth mentioning, that this estimate fails both for $p=1$ and for $p=\infty$.

The second important ingredient is the following algebraic lemma formulated in \cite[Theorem 3.2]{GSPoinc}, which directly uses the constant rank property over the algebraically closed field $\C$. Its proof crucially relies on Hilbert's Nullstellensatz (which is the abstract reason why the same argument fails for constant rank operators over $\R$). 

\begin{prop} \label{thm:NST}
Let $\A \colon C^{\infty}(\R^N,\R^d) \to C^{\infty}(\R^N,\R^l)$, $A \colon C^{\infty}(\R^N,\R^d) \to C^{\infty}(\R^N)$ be two constant-coefficient homogeneous differential operators. Suppose that $\A$ satisfies the constant rank property over $\C$. Then the following statements are equivalent:
\begin{enumerate} [label=(\alph*)]
\item \label{NST:1} there exists $s \in \N$ and a differential operator $L \colon C^{\infty}(\R^N,\R^l) \to C^{\infty}(\R^N,(\R^N)^{\otimes s})$, homogeneous of order $s$, such that \[
D^s \circ A = L \circ \A;
\]
\item \label{NST:2} $\ker_\C \A [\xi] \subset \ker_\C A[\xi]$ for all $\xi \in \C^N \setminus \{0\}$,
\item \label{NST:3} there is a finite dimensional subspace $\Q$, such that $\ker \A \subset \ker A + \Q$ (as subsets of $L^1(\Omega;\R^d)$ for any bounded Lipschitz domain $\Omega$).
\end{enumerate}
\end{prop}

In contrast to both Lemma \ref{lemma1} and Proposition \ref{thm:NST} the following statement is only a technical detail we need in the proofs of Theorem \ref{thm1} and Theorem \ref{thm2}.

\begin{lem} \label{lem:image}
Let $A \colon C^{\infty}(\R^N,\R^d) \to C^{\infty}(\R^N,\R^e)$ be a differential operator of order $k$ and $\pi$ be a polynomial with degree $\deg(\pi)$. If there exists $u \in W^{k,p}(\Omega,\R^N)$ with $A u= \pi$, there is a polynomial $\Pi$ with degree $\deg(\pi)+k$, such that $A \Pi = \pi$.
\end{lem}

\begin{proof}
First of all, we see that this is true if $\pi$ is constant. Indeed, the first property can only hold true if 
    \[
        \pi \in \spann_{\V \alpha \V=k} \Image A_{\alpha}
    \]
due to the definition of the differential operator $A$. Consequently, $\pi = \sum_{\V \alpha \V =k} A_{\alpha} v_{\alpha}$ for some appropriate $v_{\alpha} \in \R^d$. But then one obtains 
    \[
     \pi = A \left( \sum_{\V \alpha \V=k} \frac{1}{\alpha!} x^{\alpha} v_{\alpha} \right).
    \]
If $\pi$ has higher order, one writes $\pi= \pi_0+\pi_1+\ldots+\pi_s$ for $i$-homogeneous polynomials $\pi_i$. We then show that any homogeneous component can be written as $\pi_i= A \Pi_i$. Starting with $i=s$, note that $D^s \pi_s$ is constant and can be written as $D^s \pi_s = (D^s A) u$. Consequently, there is a homogeneous polynomial $\Pi_s$ of order $(s+k)$ such that $D^s \pi_s= (D^s \circ A) \Pi_s$ and hence also $\pi_s = A \Pi_s$. Arguing by an inductive argument one then obtains polynomials $\Pi_i$ with $\pi_i = A \Pi_i$, finishing the proof.
    
\end{proof}

The last important result is the connection of cancelling operators to an Bourgain-Brezis type estimate due to van Schaftingen \cite[Theorem 4]{VS_old}. In below Lemma \ref{lemma5} we prove an appropriate adjustment of this result.
Let us first recall van Schaftingen's result.
\begin{lem} \label{lemma4}
 Then for all $v \in L^1(\R^N;\R^{M_k})$ obeying $\divergence^k v =0$ and all $\phi \in C_c^{\infty}(\R^N;\R^{M_k})$ we have 
    \begin{equation} \label{VSBB}
        \left \vert \int_{\Omega} v \cdot \phi \dx \right \vert \leq \Vert v \Vert_{L^1} \Vert D \phi \Vert_{L^N}.
    \end{equation}
\end{lem}
Following \cite[Proposition 3.2]{GS}, it is possible to manipulate that statement to hold for domains that allow a solenoidal extension in $L^1$ (cf. Definition \ref{def:DFED}).
 
\begin{lem} \label{lemma41}
Let $\Omega \subset \R^N$ be a bounded divergence-free extension domain, $k \in \N$. Then for all $v \in L^1(\Omega;\R^{M_k})$ obeying $\divergence^k v =0$ and all $\phi \in C_c^{\infty}(\Omega;\R^{M_k})$ we have 
    \begin{equation} \label{VSBB2}
        \left \vert \int_{\Omega} v \cdot \phi \dx \right \vert \leq \Vert v \Vert_{L^1} \Vert D \phi \Vert_{L^N}.
    \end{equation}
\end{lem}
\begin{proof}
Due to $\Omega$ obeying the properties of Definition \ref{def:DFED}, we may find $\E v \in L^1(\R^N;\R^{M_k})$ that satisfies $\divergence^k (\E v) =0$ and $\E v - v = P_V v$ on $\Omega$. Consequently,
\[
\left \vert \int_{\Omega} v \phi \dx \right \vert \leq \left \vert \int_{\Omega} P_V v \phi \dx \right \vert + \left \vert \int_{\R^N} \E v \phi \dx \right \vert .
\]
The first term now may be estimated by using that $V$ is a finite dimensional subspace and Poincar\'{e}'s inequality, i.e. 
\[
\left \vert \int_{\Omega} P_V v \phi \dx \right \vert  \leq \Vert P_V v \Vert_{L^{N/(N-1)}} \Vert \phi \Vert_{L^N} \leq \Vert P_V v \Vert_{L^1} \Vert D \phi \Vert_{L^N}.
\]
For the second term we now may use the full-space estimate: 
\[
\left \vert \int_{\R^N} \E v \phi \dx \right \vert \leq \Vert \E v \Vert_{L^1} \Vert D \phi \Vert_{L^N}.
\]
Using boundedness of $P_V$ and of the extension operator $\E$ then yields the desired estimate \eqref{VSBB2}.
\end{proof}
The modification of Lemma \ref{lemma41} allows for linear operators that are possibly non-cancelling; but for that we need to restrict to a certain class of test-functions $\phi$.
\begin{lem} \label{lemma5}
Let $\Omega \subset \R^N$ be a divergence-free extension domain. Suppose that $\A$ satisfies the constant rank property over $\R$ and has order $k$. Moreover, suppose that 
    \[A \colon C^{\infty}(\R^N;\R^d) \to C^{\infty}(\R^N)
    \]
    is a constant coefficient operator of order $k-1$ and that there is $s \in \N_{\geq 1}$ and an operator 
        \[
        L \colon C^{\infty}(\R^N;\R^l) \to C^{\infty}(\R^N;(\R^N)^{\otimes s})
        \]
        of order $(s-1)$ such that 
    \[
        D^s \circ A = L \circ \A.
    \]
Then for all $u \in W^{k,1}(\Omega;\R^d)$ and all $\phi \in C_c^{\infty}(\Omega;(\R^N)^{\otimes s})$ we have     \[
        \left \vert \int_{\Omega} \A u \cdot L^{\ast} \phi \dx \right \vert \leq \Vert \A u \Vert_{L^1} \Vert D \circ L^{\ast} \phi \Vert_{L^N}.
    \]
\end{lem}
The main difficulty to overcome in the proof of Lemma \ref{lemma5} is, of course, the absence of the cancelling condition for the operator $\A$ - otherwise it would simply follow from van Schaftingen's result Lemma \ref{lemma4}.
\begin{proof}
The differential operator $\A$ satisfies the constant rank property over $\R$. Hence, (cf. \cite{Raita}) there exists a linear differential operator $\B \colon C^{\infty}(\R^N;\R^l) \to C^{\infty}(\R^N;\R^m)$ obeying 
    \[
        \ker \B [\xi] = \Image \A [\xi] \quad \forall \xi \in \R^N \setminus \{0\}.
    \]
We may write 
    \[
        \B v = \sum_{\V \beta \V = k_{\B}} \B_{\beta} \partial_{\beta} v, \quad \B_{\beta} \in \Lin(\R^l;\R^m)
    \]
The heart of the matter are the following two claims. Let \[
    W := \bigcap_{\xi \in \R^N \setminus \{0\} } \ker \B[\xi] = \bigcap_{\xi \in \R^N \setminus \{0\}} \Image \A[\xi]. 
\]

\textbf{CLAIM 1: \footnote{Note how this compares to \cite[Lemma 2.5]{VS} where $W=\{0\}$, i.e. $\A$ is cancelling.}}
Then there are linear maps $C_{\beta} \in \Lin (\R^m;\R^l)$ such that 
   \begin{equation} \label{aux1}
       \sum_{\V \beta \V= k_{\B}} C_{\beta} \circ \B_{\beta} = P_{W^{\perp}},
   \end{equation}
where $P_{W^{\perp}}$ is the projection onto the orthogonal complement of $W$. 

\textbf{CLAIM 2:} 
For all $\xi \in \R^N \setminus \{0\}$ and all $w \in W$ we have 
    \begin{equation} \label{aux2}
        L[\xi] (w) =0.
    \end{equation}
Let us first see, how Claim 1 and Claim 2 lead to a proof of Lemma \ref{lemma5}. Both these claims imply that 
\[
    \sum_{\V \beta \V=k_{\B}} L [\xi] \circ C_{\beta} \circ \B_{\beta} = L[\xi].
\]
Therefore, on the level of differential operators, we get 
\begin{equation} \label{equality1}
\int_{\Omega} \A u \cdot L^{\ast} \phi \dx = \sum_{\V \beta \V = k_{\B}} \int_{\Omega} (\B_{\beta} \circ \A u) \cdot (C_{\beta}^{\ast} \circ L^{\ast}) \phi \dx.
\end{equation}
As $\B$ is the annihilator of $\A$ we have \[
    \sum_{\V \beta \V = k_{\B}} \partial_{\beta} \left( \B_{\beta} \circ \A u \right) =0
\]
Employing Lemma \ref{lemma41} for $(\B_{\beta} \circ \A) u $, we get for any index $\beta$
\begin{equation} \label{equality2}
    \left \V \int_{\Omega} \B_{\beta} \circ \A u \cdot C_{\beta}^{\ast} L^{\ast} \phi \dx \right \V \leq C \Vert \B_{\beta} \circ \A u \Vert_{L^1} \Vert C_{\beta}^{\ast} L^{\ast} \phi \Vert_{L^N} \leq C \Vert \A u \Vert_{L^1} \Vert L^{\ast} \phi \Vert_{L^N}.
\end{equation}
Now combining \eqref{equality1} and \eqref{equality2} yields the result. It remains to prove both Claim 1 and Claim 2.

\textit{Proof of Claim 1:} Note that $(\xi^{\beta})_{\V \beta \V = k_{\B}}$ is a basis of the vector space of polynomials of degree $k_{\beta}$. The kernel of the map
    \[
       T \colon  w \longmapsto [\B_{\beta}(w)]_{\V \beta \V =k} \in \R^l \otimes \R^{\binom{N+k_{\B}+1}{k_{\B}}}
    \]
exactly is $W$. Therefore, $T$ is injective on $W^{\perp}$. Consequently, there are linear maps $C_\beta$, such that \[
        \sum_{\V \beta \V= k_{\B}} C_{\beta} \circ \B_{\beta} = P_{W^{\perp}},
            \]
proving Claim 1.

\textit{Proof of Claim 2:} Suppose that $w \in W$ and that there is $\xi_1 \in \R^N \setminus \{0\}$, such that $L[\xi_1](w) \neq 0$. Let $\xi_2$ be linearly independent to $\xi_1$. Then for any $t \in \R$ there exists $y(t)$ such that 
    \begin{equation}\label{def:w}
        w= \A[\xi_1+t\xi_2](y(t)).
    \end{equation}
Now we use that $L$ is a homogeneous polynomial of order $s-1$. Therefore, we can find  $\lambda_1,...,\lambda_s \in \R$ such that 
\[
       0 \neq L[\xi_1](w) = \sum_{i=1}^s \lambda_i L[\xi_1+i \xi_2](w).
\]
Consequently, 
    \[
        (L \circ \A)[\xi_1](y(0)) = \sum_{i=1}^s \lambda_i (L \circ \A)[\xi_1+ i \xi_2](y(i)).
    \]
Now applying $L \circ \A = D^s \circ A$ and setting $W_i = A[\xi_1+i\xi_2](W(i)) \in \R$ one arrives at 
    \begin{equation} \label{finale}
        0 \neq D^s[\xi_1](W_0) = \sum_{i=1}^s \lambda_i D^s[\xi_1+i\xi_2](W_i) = \sum_{i=1}^s \lambda_i W_i \lambda^{s}[\xi_1+i \xi_2](1)
    \end{equation} 
But the set of $(s+1)$ vectors $D^s[\xi_1],D^s[\xi_1+\xi_2],...,D^s[\xi_1+s \xi_2]$ is linearly \emph{independent}. This contradicts \eqref{finale}. \\
Therefore, for any $w \in W$ and any $\xi \in \R^N \setminus \{0\}$ we have $L[\xi](w) =0$, proving Claim 2.

\end{proof}
\section{Proofs of Korn- and Sobolev-type inequalities} \label{proofs}
In this section, we finally present the proofs of Theorem \ref{thm1} and Theorem \ref{thm2}. We start with the Korn-type-inequality:

\begin{proof}[Proof of Theorem \ref{thm1}:]
For `\ref{thm1:1} $\Rightarrow$ \ref{thm1:2}', it is sufficient to prove the following statement: If \ref{thm1:2} is false, then there is an \emph{infinite-dimensional} vector space $V \subset W^{k,p}(\Omega,\R^d)$, such that for all elements $u \in V \setminus \{0\}$ 
    \[
        A u \neq 0 \quad \text{and} \quad  \A u =0.
    \]
Suppose now that \ref{thm1:2} does not hold. Then there is $\xi \in \C^N \setminus \{0\}$ and $v\in \ker \A[\xi] \setminus \ker A[\xi]$. Therefore, the functions 
    \[
        u_n(x) := \mathrm{Re}\left[v e^{2 \pi i n \xi \cdot x}\right]
    \]
(where $\mathrm{Re}$ denotes, as usual, the real part of a complex number) satisfy $\A u_n=0$ and $A u_n \neq 0$. More precisely,
    \[
        A u_n (x) =  \mathrm{Re} \left[(2 \pi i n)^k A [\xi] (v) e^{2 \pi i n \xi \cdot x}\right].
    \]

As all $e^{2 \pi i n \xi \cdot x}$ are linearly independent, all $u \in \spann_{n \in \N} u_n \setminus \{0\}$ satisfy $\A u=0$, but $A u \neq 0$. The vector space $\spann_{n \in \N} u_n$ is infinite-dimensional, which finishes the proof of `\ref{thm1:1} $\Rightarrow$ \ref{thm1:2}'.

For the implication `\ref{thm1:2} $\Rightarrow$ \ref{thm1:1}' we combine Lemma \ref{lemma1} and Theorem \ref{thm:NST}. Due to Theorem \ref{thm:NST}, there exists a linear operator $L$ and $s \in \N$, such that 
    \[
    \nabla^s \circ A = L \circ \A.
    \]
Due to Lemma \ref{lemma1}, there is a linear and bounded map $P_s \colon L^p(\Omega,\R^d) \to \mathrm{Pol}_{s+1}$, such that 
    \begin{equation} \label{Necasappl}
    \Vert A u - (P_s(A u )) \Vert_{L^p} \leq \Vert D^s \circ A u \Vert_{W^{-s,p}} = \Vert L \circ \A u \Vert_{W^{-s,p}} \leq C \Vert \A u \Vert_{L^p}.
    \end{equation}
The last inequality in \eqref{Necasappl} follows from the fact that $L$ is a constant coefficient differential operator of order $s$.

The left hand side can be estimated from below by 
    \[
        \inf_{\pi \in \mathrm{Pol}_{s+1}} \Vert A u - \pi \Vert_{L^p}.
    \]
Due to Lemma \ref{lem:image}, then with $\Q$ denoting the space of $\R^d$-valued polynomials of order at most $s+k+1$
    \[
     \inf_{\Pi \in \Q}\Vert A (u - \Pi) \Vert_{L^p} \leq C \Vert \A u \Vert_{L^p}.
    \]
As $W^{k,p}(\Omega,\R^d)$ is uniformly convex and $\Q$ is a closed subspace (and so is $A (\Q)$), we find that there is a linear projection map $P_{\Q} \colon W^{k,p}(\Omega,\R^d) \to \Q$, such that \[
 \Vert A (u -P_{\Q} u) \Vert_{L^p}  =    \inf_{\Pi \in \Q}\Vert A (u - \Pi) \Vert_{L^p} \leq C \Vert \A u \Vert_{L^p}.
\]
\end{proof}

Having proven Theorem \ref{thm1}, in the same fashion we may also prove the following variant by just employing Lemma \ref{lemma1} for suitable $s$.

\begin{coro} \label{coro1}
Let $\Omega \subset \R^N$ be a bounded Lipschitz domain. Suppose that $\A$ is a constant-coefficient differential operator of order $k$ satisfying the complex constant rank property and that $A$ also has constant coefficients and order $k$. Let $1<p<\infty$ and $s \in \N$. Then the following two are equivalent:
\begin{enumerate}[label=(\alph*)]
    \item There exists a finite dimensional subspace $\Q \subset W^{s,p}(\Omega;\R^d)$ such that 
        \[
            \inf_{v \in \Q} \Vert A(u-v) \Vert_{W^{s-k,p}} \leq C_p \Vert \A u \Vert_{W^{s-k,p}}
        \]
    for all $u \in W^{k,p}(\Omega;\R^d)$;    

    \item for all $\xi \in \C^N \setminus \{0\}$ we have $\ker \A[\xi] \subset \ker A[\xi]$.
\end{enumerate}
\end{coro}

This intermediate result now allows us to prove Theorem \ref{thm2} (in the case $p>1$).

\begin{proof}[Proof of Theorem \ref{thm2}]
Note that for the exact same example as in the proof of  Theorem \ref{thm1}, the direction '\ref{thm2:1} $\Rightarrow$ \ref{thm2:2}' can be proved by contradiction. So we just need to proof the converse implication.

If $p>1$ due to a combination of Lemma \ref{lemma1} and Corollary \ref{coro1} we have:
\begin{equation} \label{step1}
\inf_{v \in \bar{Q}} \Vert A (u-v) \Vert_{L^{p\ast}} \leq C \inf_{v \in \Q} \Vert D \circ A (u-v) \Vert_{W^{-1,{p\ast}}} \leq C \Vert \A u \Vert_{W^{-1,{p\ast}}}.
\end{equation}
if $\bar{Q}= \Q +W$, where $\Q$ is taken from Corollary \ref{coro1} and $W$ is the space of constant functions. But the Sobolev embedding entails $L^p \hookrightarrow {W^{-1,{p\ast}}}$ and therefore
\[
\inf_{v \in \bar{Q}} \Vert A (u-v) \Vert_{L^{p\ast}} \leq C \Vert \A u \Vert_{L^p}.
\]
 If $p=1$ we cannot use the Sobolev embedding and instead use Lemma \ref{lemma5}. Note that we may estimate 
 \[
  \inf_{v \in \bar{Q}_s} \Vert A(u-v) \Vert_{L^{N/(N-1)}} \leq C \Vert (D^s \circ A) u \Vert_{W^{-s,N/(N-1)}}
 \]
 for any $s \in \N$ and $\bar{\Q}_s$ being polynomials of order at most $s$. As we may choose an $s \in \N$ and a differential operator $L$ of order $(s-1)$ such that $D^s \circ A = L \circ \A$, we have 
 \[
 \Vert (D^s \circ A) u \Vert_{W^{-s,N/(N-1)}} = \sup_{0 \neq \phi \in C_c^{\infty}(\Omega;(\R^N)^{\otimes s}} \int \A u L^{\ast} \phi \dx \Vert \phi \Vert_{W^{s,N}}^{-1}
 \]
 Emplyoing Lemma \ref{lemma5} we see that for any $\phi \in C_c^{\infty}(\Omega;(\R^N)^{\otimes s})$ we have 
 \[
 \left \V \int \A u L^{\ast} \phi \dx \right \V \leq C \Vert \A u \Vert_{L^1} \Vert D \circ L^{\ast} \phi \Vert_{L^N} \leq C \Vert \A u \Vert_{L^1} \Vert \phi \Vert_{W^{s,N}}
 \]
Combining these observations finally yields 
\[
\inf_{v \in \bar{Q}_s} \Vert A(u-v) \Vert_{L^{N/(N-1)}} \leq C \Vert \A u \Vert_{L^1}.
\]

\end{proof}

\bibliography{biblio.bib}
\bibliographystyle{abbrv}

\end{document}